\documentclass[12pt]{amsart}

\usepackage{amssymb}

\usepackage{enumitem}

\usepackage{graphicx}

\makeatletter
\@namedef{subjclassname@2020}{%
  \textup{2020} Mathematics Subject Classification}
\makeatother

\usepackage[T1]{fontenc}

\newtheorem{theorem}{Theorem}[section]
\newtheorem{corollary}[theorem]{Corollary}
\newtheorem{lemma}[theorem]{Lemma}

\theoremstyle{definition}
\newtheorem{definition}[theorem]{Definition}
\newtheorem{remark}[theorem]{Remark}

\newcommand{\dist}{d}

\DeclareMathOperator{\diam}{diam}

\newcommand{\C}{\mathbb{C}}
\newcommand{\D}{\mathbb{D}}
\newcommand{\N}{\mathbb{N}}
\newcommand{\R}{\mathbb{R}}

\newcommand{\eps}{\varepsilon}

\numberwithin{equation}{section}

\begin{document}


\baselineskip=17pt


\title[]{On symmetries of the IFS attractor}

\author[]{Genadi Levin}
\address{Institute of Mathematics, The Hebrew University of Jerusalem, Givat Ram,
Jerusalem, 91904, Israel}
\email{levin@math.huji.ac.il}

\date{}

\begin{abstract}
We apply some methods and technique of complex dynamics to study the set of
symmetries of attractors of holomorphic Iterated Function Systems (IFS), as well as relations between IFS sharing the same attractor.
\end{abstract}



\maketitle

\section{Introduction}


An iterated function system (IFS) is a finite collection $G=(g_1,...,g_m)$ of contractions of a complete metric space $X$.
The attractor of the IFS is a unique non-empty compact set $A\subset X$ such that $A=\cup_{i=1}^m g_i(A)$
 \cite{hutch}.


The aim of the present paper is to apply methods of complex dynamics as in \cite{le}, \cite{lp} to studying the set of
symmetries of attractors of holomorphic IFS on the complex plane (see Definition \ref{d-ifs}), as well as
families of such IFS sharing the same attractor.

Similar (and apparently closely related) problems in the dynamics of rational functions on the Riemann sphere about symmetries of the Julia set and rational functions sharing the same Julia set have been studied rather intensively, see \cite{B1, B2, be, E, Fe, le, lp}, and \cite{duj, ji} for recent important variations.

As far as we are aware, the existing literature on the above mentioned problems for IFS
is devoted almost exclusively to the case
of IFS consisting of
affine contractions of $\R^d$,
with more elaborate results
if the contractions are semilitudes (i.e., affine $f:\R^d\to\R^d$ such that $\dist(f(x),f(y))=\rho\dist(x,y)$
for some similarity factor $\rho\in (0,1)$ and all $x,y$;
the attractor of such IFS is called a self-similar set) or even homogeneous semilitudes (i.e., all having the same similarity factor),
see
\cite{deng}, \cite{elekes}, \cite{moran}, \cite{yao} and comments below.

Let us describe briefly results of the paper. We call a plane compact {\it J-like} (relative to a domain) if the set of attracting fixed points of local holomorphic symmetries of the compact is dense (on the intersection of the domain with the compact), see Definitions \ref{d-symm}-\ref{d-jlike}.
A general result allowing us to deal with the above problems is Theorem \ref{t-symm}
about symmetries of J-like compacts which says roughly speaking that any non-trivial local family of symmetries of $K$ is finite provided that
the compact $K$ is J-like and "non-laminar" (note that all Cantor $K$ are not laminar).
It closely resembles the main result of \cite{le} on symmetries of the Julia set
of a rational function of the Riemann sphere.
Namely, it turns out that Theorem 1 of \cite{le} where we prove the finiteness of the set of local holomorphic symmetries of the Julia set can be carried over, with surprisingly little changes, to a more general setting of J-like compacts.

We derive from Theorem \ref{t-symm} that any J-like non-laminar compact $K$ with no isolated points admits at most finitely many
homeomorphisms of $K$ which extend to conformal homeomorphisms to a fixed neighborhood of $K$, see Theorem \ref{t-homeo}.
(Recall that conformal means preserving angles but not necessarily orientation.)
It generalizes some results about the group of isometries of IFS attractors,
see \cite{deng}, \cite{elekes}, \cite{moran} where proofs are heavily based on the fact that isometries are affine maps.

A striking similarity between the Julia set $J$ of a rational function $f$ and the attractor $A$ of an IFS of contractions $g_1,...,g_m$ is that the set
of (repelling) fixed points of the iterates $f^n$ of $f$ and the set of (attracting) fixed points of all finite compositions
$g_{i_1}\circ...\circ g_{i_n}$ are dense in $J$ and in $A$, respectively. Although $g_i$'s are not always symmetries of $A$, in Lemma \ref{l-attrjlike} we give several sufficient conditions that they are.
It involves so-called the {\it strong separation condition (SSC)} and (apparently weaker) the {\it strong open set condition (strong OSC) with the set} $O$, see e.g. \cite{hutch}, see also definitions in Subsection \ref{ss-def}.
Lemma \ref{l-attrjlike} says, in particular, that if an IFS satisfies the strong OSC with an open set $O$ then its attractor is J-like relative every component of $O$. In particular, attractors of IFS with the SSC are J-like.
This allows us to apply Theorem \ref{t-symm} to IFSs.



Let $IFS(A)$ be a collection of IFSs from a given category (e.g. affine, semilitudes, holomorphic etc) which all have the same attractor $A$.
${\it IFS}(A)$ is well understood for some special subclasses of semilitudes: homogeneous IFS under certain separation conditions \cite{feng}, \cite{deng0},
and non-homogeneous IFS under some strong
extra conditions \cite{deng0}, \cite{yao}.
Summarising roughly these results, the semigroup ${\it IFS}(A)$ (by composition) for those subclasses of semilitudes is finitely generated.

We consider ${\it IFS}(A)$ for holomorphic IFS and ask the following very basic question:
given $m\ge 2$, under what conditions the total number of $G\in {\it IFS}(A)$ such that the number of maps $\# G$ in each $G$ does not exceed $m$ is finite?

This question makes sense and is also fundamental for the semigroup of rational functions $f$ sharing the same Julia set (where $m$ should mean the degree of $f$).
Although the answer is Yes in some cases (e.g. polynomials) it is still unknown in general whether the number of (non-exceptional) rational functions
of a given degree sharing the same Julia set is finite, see \cite{lele}.

In Theorem \ref{t-finite} we prove that a collection $\mathcal{G}\subset {\it IFS}(A)$
of holomorphic IFS with the same attractor $A$ and such that
\begin{equation}\label{int-finite}
\sup\{\# G: G\in\mathcal{G}\}<\infty
\end{equation}
is finite provided that there exists an open set
$U$ intersecting $A$ such that $g(U)$ are pairwise disjoint for all $g\in G$ and any $G\in\mathcal{G}$.
It turns out that such set $U$ always exists if all $G$ are real and the SSC holds. Thus (\ref{int-finite}) along is enough that
the collection $\mathcal{G}$ of real holomorphic IFS $G$ with the SSC sharing the same
(real) attractor is finite.
see Corollary \ref{c-finite} where we use special domains
coming from holomorphic dynamics to construct the set $U$.


Then we apply Theorem \ref{t-symm} to the problem of characterising
holomorphic IFS which have the same attractor.
In Theorem \ref{t-symmifs}, we obtain necessary and sufficient conditions for two holomorphic IFS $G=\{g\}$, $F=\{f\}$ under the SSC to share the attractor in two (equivalent) form: an infinite system of functional equations connecting iterates of maps and their inverses that hold locally, see (II) of Theorem \ref{t-symmifs}, and a finite number of functional equations that hold on neighborhoods of the attractors of $G$ and $F$, see (III) of that theorem. Corollary \ref{c-preper} says that $G,F$ share the same attractor if and only if their sets of (eventually) periodic points coincide.
See also Remark \ref{r-rat} where it is shown that if all maps of IFS $G$, $F$ happen to be local inverse branches of rational functions $\varphi, \psi$ respectively, then all these equations of Theorem \ref{t-symmifs} are reduced to any single one of them and, in fact, take the following form:
\begin{equation}\label{lp}
\varphi^{-m}\circ \varphi^M=\psi^{-n}\circ \psi^N
\end{equation}
identically on an open set, for some $0<m<M$, $0<n<N$ and some local branches $\varphi^{-m}$, $\psi^{-n}$.
This follows essentially from \cite{lp} where the latter functional equation (\ref{lp}) between rational functions $\varphi$, $\psi$ plays the major role.

In Theorem \ref{t-finitegen},
given an IFS under the SSC,
we describe the set of all possible holomorphic IFS that satisfy the SSC and share the same attractor.
There are two immediate corollaries of Theorem \ref{t-finitegen} and its proof. Corollary \ref{c-exact} says (in particular) that, given a holomorphic IFS $G$ with the SSC along with a local symmetry $h$ of the attractor of $G$ in a ball centered at a fixed point $a$ of $h$, the
following conjugacy holds in some ball around $a$:
$$h^l=g_{\bf v}\circ g_{{\bf \tilde v}}\circ g_{\bf v}^{-1},$$
for some finite compositions $g_{\bf v}$, $g_{\bf \tilde v}$ of maps of $G$ and some iterate $h^l$ of $h$.
This generalizes a result of \cite{elekes} from similitudes to holomorphic IFS.
In turn, these conjugacy relations imply
that the multiplier spectra for two IFS with the same attractor are eventually the same, see Corollary \ref{c-sp}. This also extends corresponding results
from \cite{deng}, \cite{elekes} to holomorphic IFS.

\

{\bf Notations.}
\begin{itemize}
\item Domain is an open connected subset of the plane.
\item $f'$ denotes the standard (complex analytic) derivative of a function $f$.
\item $f^n$ is the $n$-iterate of a function $f$.
\item $B(x,\rho)=\{z\in\C: |z-x|<\rho\}$.
\item $f:D\to\C$ where $D$ is a domain is called univalent if it is holomorphic (analytic) injective or, equivalently, analytic diffeomorphism.
\item For a finite set $\Lambda$ and $k\in\N$, $\Lambda^k=\{{\bf w}=(i_1,...,i_k): i_1,...,i_k\in\Lambda\}$
is the set of all {\it "words"} of the length $|{\bf w}|=k$,
$\Lambda^*=\cup_{k=1}^\infty\Lambda^k$ is the set of all finite words, and
$\Lambda^\N=\{(i_1, i_2,...,i_n,...): i_n\in\Lambda, n=1,2,...\}$
the set of infinite words.
\item $\dist(E_1,E_2)$ and $\diam E$ denote the Euclidean distance between $E_1,E_2\subset\R^2$ and diameter of $E\subset \R^2$ respectively.
\end{itemize}
{\bf Acknowledgments.} The author thanks the referee for careful reading and many comments which helped to revise the paper.
\section{Local (holomorphic) symmetries of plane compacta}\label{s-symm}
\begin{definition}\label{d-symm}
For a compact $K$ in the complex plane $\C$ and a domain $D\subset \C$ with $D\cap K\neq\emptyset$,
a holomorphic map
$H: D\to\C$ is a (local holomorphic) {\it symmetry} of $K$ on $D$ if
$$x\in D\cap K \Leftrightarrow H(x)\in H(D)\cap K.$$
A family of such symmetries $\mathcal{H}$ on a domain
$D$ is called {\it non-trivial} if it is a normal family (i.e., every sequence has a subsequence that converges uniformly on compacts) and any limit function is not a constant.
\end{definition}


Let us define a class of compact sets of the plane for which our results apply.
For a holomorphic function $H: D\to\C$, a point $b\in D$ is called an {\it attracting fixed point} if $H(b)=b$ and $0<|H'(b)|<1$.


\begin{definition}\label{d-jlike}
A compact $K\subset\C$ is {\it $J$-like} ($J$ stands for Julia) {\it relative to a domain} $U$ if the set of attracting fixed points of local holomorphic symmetries of $K$ is dense in $U\cap K$.
If $U$ is a neighborhood of $K$, then the entire $K$ is called a {\it $J$-like} compact set.
\end{definition}

Examples of $J$-like compact sets include, among others, Julia sets $J$ of rational functions (assuming $J\subset\C$) \cite{le}, \cite{lp}, and attractors of holomorphic IFS under some separation conditions, see the next Sect.


\begin{theorem}\label{t-symm}
Let a compact $K\subset\C$ with no isolated points be $J$-like relative to a domain $U$. Then
either
\begin{enumerate}
\item any non-trivial family of symmetries of $K$ defined on $U$ is finite,

or
\item any open subset of $U$ intersecting $K$ contains a simply-connected domain $\Omega$ such that $\overline{\Omega}\cap K$ is analytically diffeomorphic to the product
$I\times L\subset \C$ of the unit interval $I=[0,1]$ and a compact set $L\subset i\R$
(i.e., there is an analytic diffeomorphism $\psi$ defined in a neighborhood of $\overline{\Omega}$ such that $\psi(\overline{\Omega}\cap K)=[0,1]\times L$).
\end{enumerate}
\end{theorem}
Let us call a compact set $K$ as in the case (2) {\bf laminar} relative to a domain $U$. Consequently, we call $K$
{\bf non-laminar} if there is no domain $U$ such that $K$ is laminar relative to $U$.


Note that the case (1) holds whenever $K$ is a Cantor $J$-like compact set (Cantor sets are obviously non-laminar).
A compact $K\subset\R$ is non-laminar if and only if $K$ contains no intervals.

{\bf Proof of Theorem \ref{t-symm}} repeats almost literally the one of \cite{le}, Theorem 1 for symmetries of Julia sets
(see also Lemma 2 of \cite {lp} for a somewhat shorter account). There are two reasons for this to make it work: first, the proof in \cite{le} is local; secondly, it is based on the fact that repelling fixed points of iterates of the rational function
(equivalently, attracting fixed points of their local inverse branches) are dense in the Julia set.
So assuming there is a non-trivial infinite family of symmetries of $K$ defined in $U$
we follow the proof of \cite{le} (omitting Lemma 2 and Remark 3 there), replace local inverse branches of iterates of the rational function near their attracting fixed points which are used in the course of the proof of Theorem 1 of \cite{le} by appropriate local holomorphic symmetries of $K$, and stop every time that we arrive at the case (2) of Theorem \ref{t-symm}.

\

We derive the following corollary about the group of all locally conformal automorphisms of $K$.


\begin{theorem}\label{t-homeo}
Let $K$ be a J-like non-laminar compact on the plane with no isolated points.
Let $V$ be a neighborhood of $K$ such that every component of $V$ intersects $K$.
Then the set
$$\mathcal{I}_V=\{f:V\to\C | f(K)=K, f \mbox{ conformal injective}\}$$
is finite.

\end{theorem}
Recall that $f:U\to\C$ is conformal on an open set $U$ if it is locally injective holomorphic or anti-holomorphic on each component of $U$.
For the proof, see Subsection \ref{ss-t-homeo}.

\section{Holomorphic IFS sharing the same attractor: preliminaries and main results}\label{s-ifs}
\subsection{Holomorphic IFS}\label{ss-def}
\begin{definition}\label{d-ifs}
A (finite) {\it holomorphic IFS} is a pair $(G, \Omega)$ where $\Omega$ is a bounded domain in $\C$ equipped with the hyperbolic metric
$d_\Omega$ and $G=\{g_i:\Omega\to\Omega\}_{i\in\Lambda}$ is a finite collection of univalent maps such that
$\overline{g_i(\Omega)}\subset\Omega$ for all $i$.



\end{definition}
Since $g_i(\Omega)$ is compactly contained in $\Omega$,
it follows from the Schwarz lemma (stated as Theorem \ref{schw} below)
that each $g_i:\Omega\to \Omega$ is a strict contraction of the complete metric space $(\Omega, d_\Omega)$.
Therefore, the pair $(G, \Omega)$ is a classical IFS.

{\bf Examples}.

1. Let $g_i(z)=\alpha_i z+ b_i$ where $|\alpha_i|<1$, $b_i\in\C$.
Take any big enough disk $\Omega=B(0, R)$. Then $(G, \Omega)$ is a holomorphic IFS.
Note that the hyperbolic geometry of $B(0, R)$ tends to the Euclidean one as $R\to\infty$, uniformly on compacts of the plane.

2. Let $\varphi(z)=z^2+c$ where the complex parameter $c$ is such that $\varphi^n(0)\to\infty$ as $n\to\infty$
(i.e., $c$ is outside of the Mandelbrot set). Let $B_\infty=\{z\in\C: \varphi^n(z)\to\infty, \ n\to\infty\}$ be the basin of infinity,
$\Phi$ is the Green function of $B_\infty$ such that $\Phi(z)=\log(z)+O(1)$ as $z\to\infty$ extended by zero to the whole plane,
and $\Omega=\{z: \Phi(z)<C\}$ for some $C$ less than $\Phi(c)$ and close enough to it. Then $\Omega$ is a simply connected domain and
$\varphi^{-1}$ has two well-defined branches on $\Omega$ denoted by $g_1, g_2$ so that
the closures of $g_i(\Omega)$, $i=1,2$ are contained in $\Omega$ and are pairwise disjoint. In other words,
$(G, \Omega)$ where $G=(g_1, g_2)$ is a holomorphic IFS. Note that the attractor of this IFS coincides with the Julia set
$J=\{z: \varphi^n(z)\in\Omega, n=0,1,2,...\}$ of $\varphi$.

Holomorphic IFSs appear also naturally in holomorphic dynamics as inverse to first return maps.


\

Denote by $A=A_G$ the attractor of a holomorphic IFS $(G, \Omega)$, that is,
a unique non-empty compact set $A\subset \Omega$ such that $A=\cup_{i=1}^m g_i(A)$.
We will always assume that $A$ contains at least two points:
$$\# A\ge 2.$$



For any word ${\bf w}=(i_1,...,i_k)\in\Lambda^*$, we let $g_{\bf w}=g_{i_1}\circ...\circ g_{i_k}:\Omega\to\Omega$.
Note that $g_{\bf w}$ is a univalent (holomorphic injective) function on $\Omega$ so that there exists an inverse
univalent map $g^{-1}_{\bf w}: g_{\bf w}(\Omega)\to \Omega$. Since $\cup_{i\in\Lambda}\overline{g_i(\Omega)}$ is a compact subset of $\Omega$,
then $g_{\bf w}(x)\to A$ as $|{\bf w}|\to\infty$ uniformly in $x\in\Omega$.
Each contraction $g_{\bf w}$ of the complete space $(\Omega, d_\Omega)$,
has a unique fixed point $\beta_{\bf w}\in\Omega$, moreover, $\beta_{\bf w}\in A$.
It is easy to see that $\beta_{\bf w}$ is an attracting fixed point of $g_{\bf w}$.
The following fact is fundamental \cite{hutch}:

($\star$) {\it the set
$${\bf Per}_G=\{\beta: g_{\bf w}(\beta)=\beta, \ {\bf w}\in\Lambda^* \}$$
of fixed points of all $g_{\bf w}$ is dense in $A$}.



While there exists $s_h\in (0,1)$ such that the norm of the derivative of each $g_i:\Omega\to \Omega$ in the hyperbolic metric $d_\Omega$ is less than $s_h$, this is not necessary so in the Euclidean one. However, for every compact set $K\subset\Omega$ there exist $N\ge 1$ and $s\in (0,1)$ such that $|g'_{\bf w}(z)|\le s^n$ for all ${\bf w}\in\Lambda^n$ with $n\ge N$ and all $z\in K$.
(Proof: choose a domain $\Omega'$ such that $K\subset\Omega'\subset\overline{\Omega'}\subset\Omega$ and $\cup_i g_i(X)\subset \Omega'$;
then there exist $C>0$ and $q\in (0,1)$ such that
$|g'_{\bf w}(z)|\le C q^n$ for all ${\bf w}\in\Lambda^n$ and all $z\in \overline{\Omega'}$. This is because the hyperbolic metric $d_\Omega$ and the Euclidean metric are equivalent being restricted to $\overline{\Omega'}$.
Then take $s\in(q, 1)$ and $N$ big enough.)


\subsection{Separation conditions and the associated inverse dynamics}\label{ss-sep}
Recall \cite{hutch} that IFS consisting of contractions $\{g_i\}$ and having the attractor $A$ satisfies the {\bf strong separation condition (SSC)} if
$$g_i(A)\cap g_j(A)=\emptyset \mbox{ for all } i\neq j,$$
and the {\bf strong open set condition (strong OSC)} with $O$ if $O$ is an open set
such that
$$A\cap O\neq\emptyset, \  g_i(O)\subset O \mbox{ for all } i, \ g_i(O)\cap g_j(O)=\emptyset \mbox{ for all } i\neq j.$$
Note that then \cite{hutch} $A\subset\overline{O}$.

The SSC implies the strong OSC, with the set $O$ to be a small neighborhood of $A$\footnote{For conformal IFS, the strong OSC and the OSC
(for which the condition $A\cap O\neq\emptyset$ is not required) are equivalent, with possibly different open sets, see \cite{peres}.}. The attractor of an IFS with at least two maps that satisfies the SSC is necessary a Cantor set with no isolated points.

Let us call a pair $(G, V)$ a {\bf box-like} restriction of a holomorphic IFS $(G, \Omega)$ where $G=\{g_i: i\in\Lambda\}$
if $V\subset\Omega$ is a neighborhood of the attractor $A_G$ such that the sets $\overline{g_i(V)}$ for $i\in\Lambda$ are pairwise disjoint and are contained in $V$.

If the SSC holds for an IFS $G=(g_1,...,g_m)$ with the attractor $A=A_G$ then the following associated {\bf inverse map (dynamical system)} $\varphi=\varphi_G: A\to A$ is well-defined:
$$\varphi|_{g_i(A)}=g_i^{-1}, i=1,....m.$$
If $(G, V)$ is a box-like restriction of $(G, \Omega)$ then $\varphi: A\to A$ extends to a holomorphic $m$-covering map of $U=\cup_{i\in\Lambda}(V)$ onto $V$ which we call a {\bf complex inverse map} $\varphi: U\to V$.
(In Example 2 above the whole IFS $(G, \Omega)$ is a box-like while $\varphi: g_1(\Omega)\cup g_2(\Omega)\to\Omega$ is the complex inverse map.)


Recall that, given a map $f:X\to X'$, a point $x\in X$ is pre-periodic if $f^p(x)=f^q(x)$ for some minimal $0\le p < q$ (assuming all iterates
$f(x),...,f^q(x)$ are well-defined); in other words, $x$ is pre-periodic if and only if the sequence of iterates $\{f^n(x)\}_{n\ge 0}$ is well-defined
and is a finite set. (Note that $p=0$ corresponds to the case when $x$ is a periodic point of $f$.)

\begin{lemma}\label{l-ifsinv}
\begin{enumerate}\label{repel}
\item $\varphi:A\to A$ has degree $m$; for every $n$ and every ${\bf w}\in\Lambda^n$,
$$\varphi^n|_{g_{\bf w}(A)}=g_{\bf w}^{-1},$$
in particular, $\varphi^n(x)=x$  if and only if $x$ is a fixed point of $g_{\bf w}$ for some ${\bf w}\in\Lambda^n$,
\item suppose the IFS $(G,\Omega)$ is holomorphic. Then there exists a box-like restriction $(G, V)$ of $(G, \Omega)$ such that
each connected component of $V$ intersects $A$; moreover, for every $\eps>0$ one can choose the domain $V$ such that every component of $V$ has the Euclidean diameter less than $\eps$,
\item if $\varphi:U\to V$ is a complex inverse map, then $A=\{x\in U: \varphi^n(x)\in U, n=0,1,...\}$.
In particular, any pre-periodic point of $\varphi:U\to V$ lies in $A$ and $A$ is the closure of the set of all periodic points of $\varphi:U\to V$.
\end{enumerate}
\end{lemma}
Item (1) is immediate. For (2), we can let $V=V_G:=\{z\in\Omega: d_\Omega(z,A)<\sigma\}$ be the $\sigma$-neighborhood of $A$ w.r.t. the
hyperbolic metric $d_\Omega$ on $\Omega$, for $\sigma>0$ smaller than $1/2 \min\{d_{\Omega}(g_i(A),g_j(A)): i\neq j\}$.
Then $V$ is as required because all $g_i$ are contractions in the metric $d_\Omega$. Now, for every $n\in\N$, if $V_n:=\cup_{{\bf w}\in\Lambda^n}g_{\bf w}(V)$, then $(G, V_n)$ is again a box-like restriction.
On the other hand, as $g_{\bf w}(x)\to A$ as $|{\bf w}|\to\infty$ uniformly in $x\in\Omega_G$ and since $A$ is a Cantor set,
given $\eps>0$, there exists $n$ such that the diameter of each component of $V_n$ will be less than $\eps$. This proved (2).
(3) follows because
$$\{x: \varphi^n(x)\in U, n=1,2,...\}=\cap_{n=1}^\infty\cup_{{\bf w}\in\Lambda^n} g_{\bf w}(V).$$
\subsection{Symmetries of the attractor under separation conditions}\label{ss-symm}
It is easy to see that for a holomorphic IFS $(G, \Omega_G)$ with the SSC, there exists $\delta>0$ such that any finite composition $g_{\bf w}$ is a local symmetry of the attractor $A$ on any ball $B(a, \delta)$, $a\in A$.

Let us introduce a radius $\rho_G$ which quantifies this property. First, let $\rho_h>0$ be the minimal hyperbolic (i.e. $d_\Omega$-)distance between all (pairwise disjoint) compact sets $g_k(A)$.
Then, for any $a\in A$ and the hyperbolic ball $B_h(a,\rho_h):=\{z: d_\Omega(z,a)<\rho_h\}$, all $g_{\bf w}$ are symmetries on $B_h(a,\rho_h)$.
Indeed, for each $g_i$, $i\in\Lambda$, as it contracts the hyperbolic metric, $g_i(B_h(a,\rho_h))\subset B_h(g_i(a), \rho_h)$
while, by the SSC and the definition of $\rho_h$, the letter ball is disjoint with all other $g_j(A)$, $j\neq i$. As each $g_i:\Omega\to g_i(\Omega)$ is a homeomorphism and by the SSC, $x\in A$ if and only if $g_i(x)\in A$. Therefore, every $g_i$
is a local holomorphic symmetry on each hyperbolic ball $B_h(a, \rho_h)$ and the claim follows by induction along with the contraction of maps $g_i$ w.r.t. $d_\Omega$. Then we can take $\rho_G$ to be the maximal $r>0$ such that the Euclidean ball
$B(a, r)$ is contained in the hyperbolic ball $B_h(a, \rho_h)$, for all $a\in A$.

If all maps $g_k$ of $G$ happen to be linear contractions, one can simply define
$\rho_L$ to be the minimal Euclidean distance between all compact sets $g_k(A)$. Then all maps $g_{\bf w}$
will be local symmetries of $A$ on every Euclidean ball $B(a, \rho_L)$, $a\in A$.

The following claim is crucial for the rest of the paper.
\begin{lemma}\label{l-attrjlike}
Let $A$ be the attractor of a holomorphic IFS $(G, \Omega)$ where $G=(g_1,...,g_m)$.

1. Assume the SSC holds. Then $A$ is a J-like Cantor set.

2. Assume there exists an open set $U\subset \Omega$ such that $A\subset\overline{U}$, $A\cap U\neq\emptyset$ and $g_i(U)$ for $i=1,...,m$ are pairwise disjoint\footnote{Inclusions $g_i(U)\subset U$ are not assumed here.}. Then every $g_i$ is a symmetry of $A$ on each component $D$ of $U$ that intersects $A$.

3. Assume that the strong OSC holds with a set $O$.
Then any finite composition $g_{\bf w}$ of $g_i$'s is a symmetry of $A$ on each component $D$ of $O$ and $A$ is J-like relative to the domain $D$.
\end{lemma}
\begin{proof}


1. Assuming the SSC, every $g_{\bf w}$ is a local symmetry of $A$ on the ball
$B(\beta_{\bf w}, \rho_G)$. By ($\star$), points $\beta_{\bf w}$ are dense in $A$. Hence, $A$ is J-like. $A$ is totally disconnected  because all $g_i$ are strict contractions and the SSC holds.

2.
As $A\subset \overline{U}$ and $g_1(U)$ is disjoint with
all others $g_j(U)$, then $g_1(U)$ is disjoint with $g_j(\overline{U})$, hence, with $g_j(A)$, $j=2,...,m$. Therefore, $g_1(D)\cap A=g_1(D\cap A)$, which implies that $g_1$ is a symmetry on $D$.

3. Define $U_0=O, U_1=G(U_0),..., U_{k-1}=G(U_{k-2})$ where $G(E):=\cup_{i=1}^m g_u(E)$. Note that $U_0\supset...\supset U_{k-1}$.
Then every open set $U_r$, $r=0,...,k-1$,
satisfies the conditions of the open set $U$ of part 2. Therefore, all $g_i$ are symmetries on any component of any $U_r$.
Now, let $g_{\bf w}=g_{i_1}\circ...\circ g_{i_k}$.
Then $g_{i_k}$ is a symmetry on $D$, $g_{i_{k-1}}$ is a symmetry on $g_{i_k}(D)$ etc. This shows that $g_{\bf w}$ is a symmetry on $D$. The fact ($\star$) finishes the proof that $A$ is J-like relative to the domain $D$.
\end{proof}
Note however that any infinite sequence of such pairwise different maps $g_{\bf w}$ is a trivial family of symmetries because all their limit functions are constants (points of the attractor).


\subsection{Main results}

First, we study how many IFS can share the same attractor. Of course all iterates of a given IFS have the same attractor so we restrict ourself to classes $\mathcal{G}$ of IFS $G$ such that
\begin{equation}\label{sup}
\sup\{\# G: G\in\mathcal{G}\}<\infty.
\end{equation}


\begin{theorem}\label{t-finite}
Suppose that $A\subset\C$ is a compact without isolated points, $V$ is a connected neighborhood of $A$,
$U$ is an open set such that $A\subset\overline{U}\subset V$, $A\cap U\neq \emptyset$, finally, $A$ is not laminar relative to
at least one component of $U$ meeting $A$.
Let $\mathcal{G}$ be a collection of
holomorphic IFS $(G, \Omega_G)$ that satisfy condition (\ref{sup}) as well as the following conditions:
\begin{enumerate}
\item $A_G=A$ and $V\subset \Omega_G$,
\item for every $G\in\mathcal{G}$, the sets $g(U)$ for all $g\in G$ are pairwise disjoint,
\item there exists $G_*\in\mathcal{G}$ such that either $G_*$ satisfies the SSC or $\cup_{g\in G_*}g(U)\subset U$.
\end{enumerate}
Then the collection
$$\{g|_V: g\in G \mbox{ for some } G\in\mathcal{G}\}$$
contains at most finitely many different functions.
\end{theorem}

\begin{remark}\label{after-t-finite}


Conditions (2)+(3) can be replaced by the following single (though stronger in general) condition:
all $G\in \mathcal{G}$ satisfy the strong OSC with the same open set.

\end{remark}

Let us call a holomorphic IFS $(G,\Omega_G)$ {\it real} if the domain $\Omega$ is symmetric w.r.t. $\R$ and
$g(\bar z)=\overline{g(z)}$ for all $g\in G$ and $z\in\Omega$.

\begin{corollary}\label{c-finite}

Given a Cantor set $A\subset\C$ and a connected neighborhood $V$ of $A$,
let $\mathcal{G}$ be a collection of
holomorphic IFS $(G, \Omega_G)$ as follows:
\begin{enumerate}
\item $A_G=A$ and $V\subset \Omega_G$,
\item either of the two conditions holds:

(a) there exists $r>0$ such that for every $G\in\mathcal{G}$ and
any two different $g, \tilde g\in G$, $\dist(g(A), \tilde g(A))\ge r$ (in particular, the SSC holds),

(b) $A\subset \R$, $A\subset V\cap\R$ and $\mathcal{G}$ consists of real IFS that satisfy the SSC; furthermore,
condition (\ref{sup}) holds.
\end{enumerate}
Then the collection
$$\{g|_V: g\in G \mbox{ for some } G\in\mathcal{G}\}$$
is finite.
\end{corollary}

\begin{remark}\label{r-colo-2b}
The requirement in (2b) that all $G$ of the class $\mathcal{G}$ satisfy the SSC can be weakened as follows:
there is an open set $O_\R$ of the real line
such that all $G\in\mathcal{G}$ being restricted
to $\R$ satisfy
the strong OSC with the same set $O_\R$. In this case, choose $\theta>0$ small and let $U$ to be the union
of $\Omega_\theta(J)$ over all connected components $J$ of $O_\R$, see Subsection \ref{ss-s} for definition of $\Omega_\theta(J)$. Then the proof of Corollary \ref{c-finite} as in Sect \ref{proofs} holds with this new $U$.

\end{remark}
By a simple area consideration, it is easy to see that
for every $G$ as in item (2a) of Corollary \ref{c-finite},
\begin{equation}\label{2a}
\# G\le \left(\frac{1+r/(2\diam A)}{r/(2\diam A)}\right)^2.
\end{equation}
So it would be interesting to know if the condition (2a) can be replaced only by the SSC and the bound (\ref{sup})\footnote{Probably, not}.

\

All other main results are about relations between IFS sharing an attractor, cf. \cite{le}, \cite{lp}.
The following is a central theorem of the paper. It states three conditions (I)-(III) where (I) is that two holomorphic IFS $G$ and $F$ with the SSC have the same attractor and (II)-(III) are equivalent to (I) and are stated via functional equations between maps of $G$ and $F$. Note that (II) consists of {\it infinitely many} functional equations that hold
on a neighborhood of a single point while (III) consists of only {\it finitely many} functional equations that hold altogether on some neighborhoods of attractors.
\begin{theorem}\label{t-symmifs}
Let $(G, \Omega_G)$, $(F, \Omega_F)$  be two finite holomorphic IFS on $\C$ with the SSC
where
$G=\{g_i\}_{i\in\Lambda}$ and
$F=\{f_j\}_{j\in\tilde\Lambda}$.
The following conditions (I)-(III) are equivalent:
\begin{enumerate}
\item [(I)] $G$ and $F$ share the same attractor $A=A_G=A_F$,
\item [(II)]

there exist a finite collection of words ${\bf w^k}\in\Lambda^*$, ${\bf v^k}\in\tilde\Lambda^*$,
$k=1,2,...,k_0$, for some $k_0\in\N$, and there exists $m_0\in\N$
such that the following infinite sequences of functions equations (G)-(F) are satisfied:

(G) for some ball $B_G\subset\Omega_G$, for every $m\ge m_0$ and every ${\bf w}\in\Lambda^m$ there exist ${\bf v}\in\tilde\Lambda^*$ and an index $k\in\{1,...,k_0\}$ such that
\begin{equation}\label{func}
f_{\bf v}\circ f_{{\bf v^k}}^{-1}=g_{\bf w}\circ g_{{\bf w^k}}^{-1} \mbox{ on } g_{{\bf w^k}}(B_G),
\end{equation}
and both hand-sides are well-defined.

(F) the same holds after switching $G$ and $F$ (and obvious change of notations).

Moreover, if $A=A_G=A_F$ then one can choose $B_G=B_F$ to be a neighborhood of a point of $A$.

\item [(III)] there exist: (i) box-like restrictions $(G, V_G)$ and $(F, V_F)$ of $(G, \Omega_G)$ and $(F, \Omega_F)$ respectively (see Subsection \ref{ss-sep} for definition),
(ii) a finite collection of words ${\bf t^k}\in\Lambda^*$, ${\bf u^k}\in\tilde\Lambda^*$,
$k\in\{1,...,K\}$ for some $K\in\N$, and some $M\in\N$ as follows. Let $M_*=\max\{|{\bf t^k}|, |{\bf u^k}|: k=1,...,K\}$. Then $M\ge M_*+1$
and the following finite number of functional equations
($G_{fin}$)-($F_{fin}$) are satisfied:

($G_{fin}$) for every component $D$ of $V_G$ and every ${\bf t}\in\Lambda^M$ there exists
a word
${\bf u}={\bf u}(D,{\bf t})\in\tilde\Lambda^*$,
$|{\bf u}|\ge M_*+1$,
and $k\in\{1,...,K\}$ such that
$g_{{\bf t^k}}(D\cap A_G)\subset f_{{\bf u^k}}(V_F)$ and
\begin{equation}\label{func-finite}
f_{\bf u}\circ f_{{\bf u^k}}^{-1}=g_{\bf t}\circ g_{{\bf t^k}}^{-1} \mbox{ on } g_{{\bf t^k}}(D\cap A_G).
\end{equation}

($F_{fin}$) the same holds after switching $G$ and $F$ (and obvious change of notations).

\end{enumerate}
\end{theorem}
\begin{remark}
Let us rewrite the system of equations (\ref{func-finite}) as follows.
First, every eq. of (\ref{func-finite}) is equivalent to:
$f_{{\bf u^k}}\circ f_{\bf u}^{-1}=g_{{\bf t^k}}\circ g_{\bf t}^{-1}$ on $g_{\bf t}(D\cap A_G)$.
Now using the associated to $G$, $F$ complex inverse maps
$\varphi=\varphi_G:U_G\to V_G$, $\psi=\varphi_F:U_F\to V_F$,
and that $\varphi|_{g_{\bf w}(V_G)}=g_{\bf w}^{-1}$ for each ${\bf w}\in\Lambda^*$, and similar for $\psi$,
the condition ($G_{fin}$) turns into the following:
for every component $D$ of $V_G$ and every ${\bf t}\in\Lambda^M$
there exist $k\in\{1,...,K\}$ and ${\bf u}={\bf u}(D,{\bf t})\in\tilde\Lambda^*$, $|{\bf u}|\ge M_*+1$,
such that $g_{\bf t}(D\cap A_G)\subset f_{\bf u}(V_F)$ and
\begin{equation}\label{eq-inv}
\psi^{-|{\bf u^k}|}\circ \psi^{|{\bf u}|}(x)=\varphi^{-|{\bf t^k}|}\circ \varphi^M(x) \mbox{ for every } x\in g_{\bf t}(D\cap A_G),
\end{equation}
where $\varphi^{-|{\bf t^k}|}=g_{{\bf t^k}}$,
$\psi^{-|{\bf u^k}|}=f_{{\bf u^k}}$ are inverse branches of
$\varphi^{|{\bf t^k}|}$, $\psi^{|{\bf u^k}|}$ respectively.
Notice a similarity between (\ref{eq-inv}) and functional equations of \cite{lp} (see (\ref{lp}) of the Introduction).
\end{remark}
Remind that ${\bf Per}_G=\{\beta: g_{\bf w}(\beta)=\beta, \ {\bf w}\in\Lambda^* \}$ and let
$${\bf Prep}_G={\bf Per}_G\cup\{g_{\bf w}(x): {\bf w}\in\Lambda^*, x\in {\bf Per}_G\}.$$
Note that
${\bf Prep}_G=\{x\in A_G: \varphi^P(x)=\varphi^Q(x), \ 0\le P<Q\}$.
In other words, ${\bf Prep}_G$ is the set (pre-)periodic points of $\varphi: U_G\to V_G$.
\begin{corollary}\label{c-preper}
Two holomorphic IFS $(G,\Omega_G)$, $(F,\Omega_F)$ that satisfy the SSC have the same attractor if and only if
${\bf Prep}_G={\bf Prep}_F$.
\end{corollary}
In one direction ("two sets coincide$\rightarrow$attractors coincide") it is immediate
because sets ${\bf Prep}_G$, ${\bf Prep}_F$ are dense in $A_G$, $A_F$ respectively.
The proof of another direction can be found in Subsect \ref{ss-main}.
\begin{remark}\label{r-rat}
One might wonder when the equations of (II) imply the equations of (III) by analytic continuations. This is so if, for example, there exists a single domain
in $\Omega_G\cap\Omega_F$ that contains
balls $B_G,B_F$ and some neighborhoods of $A_G, A_F$. The most apparent case is probably the following.
Assume that the inverse maps $\varphi=\varphi_G$, $\psi=\varphi_F$
extend to rational functions of degrees $m=\# G$, $d=\# F$ respectively. It then follows from Lemma \ref{l-ifsinv} that Julia sets of $\varphi, \psi$ are precisely the attractors $A_G$, $A_F$  and $\varphi, \psi$ are hyperbolic rational maps
(i.e., uniformly expanding on their Julia sets).
Take a single equation (\ref{func})
and write it in the form
$\varphi^{-|{\bf w^k}|}\circ \varphi^{|{\bf w}|}=\psi^{-|{\bf v^k}|}\circ \psi^{|{\bf v}|}$
identically on some open set $Z_{\bf w}$, for some local branches $\varphi^{-|{\bf w^k}|}$, $\psi^{-|{\bf v^k}|}$.
and some choice of ${\bf w}, {\bf w^k}, {\bf v}, {\bf v^k}$ such that $|{\bf w}|>|{\bf w^k}|$, $|{\bf v}|>|{\bf v^k}|$.
Then, by Theorem 3(2) of \cite{le} or Theorem B of \cite{lp} this functional equation along is equivalent to $\varphi, \psi$ having the same Julia set (i.e., $G$ and $F$ having the same attractor). Thus
in the case when $G$, $F$ are inverse branches of rational functions
the infinite sequence of equations in the condition (II) of Theorem \ref{t-symmifs}
is reduced to any single equation of the sequence and also is equivalent to any eq. of (III).
\end{remark}
Let $(F, \Omega_F)$, $F=\{f_j\}_{j\in\tilde\Lambda}$, be a finite holomorphic IFS with the SSC, and $A=A_F$.
In the next statement we describe all possible holomorphic IFS with the SSC which can have the same attractor $A$. For every pair
${\bf v}, {\bf \tilde v}\in\tilde\Lambda^*$ of finite words, let
$$R_{{\bf v},{\bf \tilde v}}:=f_{\bf v}\circ f_{{\bf \tilde v}}\circ f_{\bf v}^{-1}.$$
It maps the domain $f_{\bf v}(\Omega_F)$ into itself and has a unique fixed point there.
Roughly speaking, the set of maps $R_{{\bf v}, {\bf \tilde v}}$ determine all other IFS with the same attractor:
\begin{theorem}\label{t-finitegen}
First, for each $l\in\N$ and each $R_{{\bf v},{\bf \tilde v}}$, there exist precisely $l$ different local (i.e, defined and holomorphic in a neighborhood of the fixed point of $R_{{\bf v},{\bf \tilde v}}$) solutions $g^{(r)}_{{\bf v},{\bf \tilde v}, l}$, $r=1,...,l$, of the equation $g^l=R_{{\bf v},{\bf \tilde v}}$ (where $g^l$ is the $l$-iterate of $g$). If now $(G,\Omega_G)$ is any other holomorphic IFS with the SSC and the same attractor $A$, then there exists $K\in\N$
such that each $g_i\in G$ is a holomorphic continuation of one of the local functions
$$\{g^{(r)}_{{\bf v},{\bf \tilde v}, l}: {\bf v},{\bf \tilde v}\in\tilde\Lambda^*,
l=1,...,K, r=1,...,l\}.$$
\end{theorem}
The proof implies:
\begin{corollary}\label{c-exact}
Given the attractor $A$ of a holomorphic IFS $G$ with the SSC, and some local symmetry
$h: B(a,r)\to\C$ of $A$ such that $a$ is an attracting fixed point of $h$, there exist $l\in\{1,...,K\}$ where $K$ depends only
on $r$ and $G$, and
some ${\bf v}, {\bf \tilde v}\in\Lambda^*$ such that
$$h^l=g_{\bf v}\circ g_{{\bf \tilde v}}\circ g_{\bf v}^{-1}$$
on some ball around the point $a$.
\end{corollary}
Here is another immediate corollary where we use the following
\begin{definition}\label{sp}
Given a finite holomorphic IFS $G=\{g_i\}_{i\in\Lambda}$ on the plane, we define the {\it multiplier (or dynamical) spectrum}
$\sigma(G)$ of $G$ as the collection of derivatives $g'_{\bf w}(\beta_{\bf w})$ over all finite words ${\bf w}\in\Lambda^*$.
\end{definition}
Note that $\lambda\in \sigma(G)$ implies $\lambda^n\in \sigma(G)$ for all $n\in\N$, and if $G=\{g_i(z)=\lambda_i z+ b_i\}_{i\in\Lambda}$ is a finite set of linear contractions, then
$\sigma(G)$ is a semigroup (w.r.t. the multiplication) generated by $\{\lambda_i\}_{i\in\Lambda}$.
\begin{corollary}\label{c-sp}
If $G$, $F$ are two holomorphic IFS that satisfy the SSC and share the same attractor, then
for every $\lambda\in \sigma(G)$ there is $l\in\N$ such that $\lambda^l\in \sigma(F)$, and the same holds after switching $G$ and $F$.
Here, all $l$ are uniformly bounded by a finite number which depends only on $F$ and $G$.
\end{corollary}
\begin{remark}\label{r-diff}
Suppose that attractors $A_G$, $A_{\tilde G}$ of two holomorphic IFS $(G,\Omega_G)$ and $(\tilde G,\Omega_{\tilde G})$ are such that
$h(A_G)=\tilde A$ for some analytic diffeomorphism $h:U\to \tilde U$ between neighborhoods $U$, $\tilde U$ of $A_G$, $A_{\tilde G}$. Let $f_j=h^{-1}\circ \tilde g_j\circ h$ where $\tilde G=\{\tilde g_j\}$. Then maps $f_j$ are defined on an appropriate neighborhood of $A_G$
and all results of this section with minor modifications apply to $G=\{g_i\}$ and $F=\{f_j\}$. Details are left to the reader.
\end{remark}
\section{Some results in complex analysis}
Proofs of the main results are based on Theorem \ref{t-symm} and several facts of geometric complex analysis.
\subsection{The Koebe distortion theorem}\label{ss-k}
The first one is the Koebe Distortion Theorem (see e.g. \cite{goluzin}):
\begin{theorem}\label{koebe}

1. For any univalent function $\phi: B(0, R)\to\C$ such that $\phi(0)=0$, and all $t\in[0,1)$,
$$B(0,\frac{t}{4}R|\phi'(0)|)
\subset\phi(B(0,tR))\subset B(0, \frac{t}{(1-t)^2}R|\phi'(0)|).$$
In particular, the family of univalent functions $\phi:B(0,R)\to\C$ such that $|\phi(0)|\le C$
and $1/C\le |\phi'(0)|\le C$ for a fixed $C$ and all $\phi$ is compact (where the convergence is uniform on compacts in $B(0,R)$).


2. Given a domain $D$ and a compact $K\subset D$, there exist $\lambda>1$ such that, for every univalent
$\phi: D\to\C$ and all $x,y\in K$:
$$\frac{|\phi'(x)|}{|\phi'(y)|}\le \lambda.$$

\end{theorem}
In fact, part 2 is a consequence of part 1.
\subsection{Schwarz's lemma and the slit complex plane}\label{ss-s}
A domain $U$ of the complex plane is hyperbolic if $\#\partial U > 2$.
Every hyperbolic domain $U$ carries a unique {\it hyperbolic metric} $d_U(z) |dz|$, i.e. such that some (hence, any) unbranched holomorphic covering map of the unit disk $\D$ onto $U$ is a local isometry in the hyperbolic metric $|dw|/(1-|w|^2)$ on $\D$ and the metric $d_U(z) |dz|$ on $U$.
We have the following version of the classical Schwarz Lemma see e.g. \cite{goluzin}:
\begin{theorem}\label{schw}
Every holomorphic map $f: D\to D'$ between two hyperbolic domains is a weak contraction in the hyperbolic metrics of $D, D'$.
It is an isometry if and only if $f:D\to D'$ is a holomorphic diffeomorphism.
\end{theorem}

For an open interval $I=(a,b)\subset \R$ and an angle $\theta\in (0,\pi)$, let $D_\theta(I)$ be the Euclidean disk uniquely defined by the following conditions: $D_\theta(I)\cap \R=I$ and angles of intersection of $\partial D_\theta(I)$ with $\R$ at the points $a,b$ are $\theta$ and $\pi-\theta$ respectively. A domain $\Omega_\theta(I)$ is said to be the union of $\{z\in D_\theta(I): \Im(z)>0\}$, its symmetric w.r.t. $\R$, and the common base interval $I$. Note that $\Omega_\theta(I)$ shrink to the interval $I$ as $\theta\to 0$.
Domains $\Omega_\theta(I)$ appear in the one-dimensional smooth dynamics since \cite{Su} thanks to the following observation (see e.g. \cite{MS}):
\begin{lemma}\label{fact-sull}
There exists a continuous strictly increasing function $\kappa: [0, \pi)\to [0, \infty)$ with $\kappa(0)=0$ as follows.
Let $\C_I:=\C\setminus (\R\setminus I)$ be the complex plane with two real slits
from $\infty$ to the end points of $I$ endowed with the hyperbolic metric $d_{\C_I}$. For every $\theta\in (0,\pi)$,
$$\Omega_\theta(I)=\{z\in\C_I: d_{\C_I}(z, I)<\kappa(\theta)\}.$$
\end{lemma}
For the proof of Corollary \ref{t-finite}, we use Lemma \ref{l-sull} below
which is a consequence of Lemma \ref{fact-sull} and Theorem \ref{schw}.
Given a simply-connected domain $W$ containing a real interval $I$ and symmetric w.r.t. $\R$ we denote
$W_I:=W\cap \C_I$.
\begin{lemma}\label{l-sull}
Let $I\subset \R$ be an open interval and $V$ be a simply connected domain such that $V$ is symmetric w.r.t. $\R$ and $\overline{I}\subset V$. Then there exists two angles $0<\theta_*<\theta_*'<\pi/2$ as follows. For every $\theta\in (0, \theta_*)$ there exists $\theta'\in (0, \theta_*')$
such that,
for every univalent map $g: V\to\C$ which is real, i.e., $g(\bar z)=\overline{g(z)}$ for all $z\in V$, we have:
$$g(\Omega_\theta(I))\subset \Omega_{\theta'}(g(I)),$$
where both hand-sides are well-defined.
\end{lemma}
\begin{proof}
First, after pre- and postcomposing $g$ by linear maps preserving the real line, one can assume from the beginning the set of maps
$\mathcal{F}=\{g\}$ under consideration consists of real univalent $g:V\to\C$  such that
$I=(0,1)=g((0,1))$. (Note by the way that the identity map is in $\mathcal{F}$.)
The Koebe distortion theorem then implies that $\mathcal{F}$ is compact, where the convergence is uniform convergence on
compacts in $V$. In turn, the compactness implies that there exists $\eps>0$ such that, for all $g\in\mathcal{F}$, $g(V)$ contains the Euclidean
$2\eps$-neighborhood of $[0,1]$. Hence, if $U$ is the $\eps$-neighborhood of $[0,1]$, then
\begin{equation}\label{comp}
U\subset g(V) \mbox{ and } \dist(\partial g(V), U)\ge \eps \mbox{ for all } g\in\mathcal{F}.
\end{equation}
Let us show that there exists $c_*>1$ such that
\begin{equation}\label{dist}
1\le \frac{d_{g(V_I)}(z)}{d_{\C_I}(z)}\le c_* \mbox{ for all } z\in U_I \mbox{ and all } g\in\mathcal{F}.
\end{equation}
Indeed, the left hand-side inequality holds because $g(V)_I\subset \C_I$. For the other inequality, we use that $1/(2\dist(z,\partial D))\le d_D(z)\le 2/\dist(z,\partial D)$ for any
simply-connected domain $D\neq\C$ and any $z\in D$ \cite{BM}.
This universal bound coupled with the fact that $\partial g(V_I)\cap\R\subset \partial \C_I$  and with (\ref{comp}) imply (\ref{dist}).

Now, fix $\theta_*'>0$ small enough so that $\Omega_{\theta_*'}(I)\subset U$ and, for every $\theta'\in (0, \theta_*')$, define $\theta=\kappa^{-1}(\kappa(\theta')/c_*)$
where $\kappa$ is the function appeared in Lemma \ref{fact-sull}. As $0<\theta'\le \theta_*'$, also $\Omega_\theta(I)\subset U$.
Introduce the following notation: $B^D_r$ is the $r$-neighborhood of the interval $I$ in the hyperbolic metric of a domain $D$ that contains $\overline{I}$.
Using (\ref{dist}), we obtain first that
$$\Omega_\theta(I)=B^{\C_I}_{\kappa(\theta)}\subset B^{V_I}_{c_*\kappa(\theta)}.$$
Then, using that $g:V_I\to g(V_I)$ is a hyperbolic isometry and again (\ref{dist}), we have:
$$g(\Omega_\theta(I))\subset g(B^{V_I}_{c_*\kappa(\theta)})=
B^{g(V_I)}_{c_*\kappa(\theta)}\subset B^{\C_I}_{c_*\kappa(\theta)}=
\Omega_{\kappa^{-1}(c_*\kappa(\theta))}(I)=\Omega_{\theta'}(I).$$
\end{proof}
\begin{remark} Note that $\theta'>\theta$ so that $g(\Omega_\theta(I))$ is not necessary contained in $\Omega_\theta(I)$.
On the other hand, if $\diam g(I)/\diam I$ is small enough (depending on $\theta'$), then $\overline{g(\Omega_\theta(I))}\subset \Omega_\theta(I)$, and one can apply the same $g$ again.
Using the explicit formula for the function $\kappa$, e.g., \cite{MS}, one can see that $\theta'\sim c_*\theta$ as $\theta\to 0$.
\end{remark}
\section{Proofs of the main results}\label{proofs}
\subsection{Theorem \ref{t-homeo}\label{ss-t-homeo}}
First, if $h: D\to\C$ is anti-holomorphic, then $h^*:D\to \C$ defined as $h^*(z)=\overline{f(z)}$ is holomorphic.
If $h$ is either holomorphic or anti-holomorphic, denote $||h'||(z)=|h'(z)|$ in the former case and $||h'||(z)=|(h^*)'(z)|$
in the latter one.

Recall that $\mathcal{I}_V=\{f:V\to\C | f(A)=A, f \mbox{ conformal injective}\}$.
\begin{lemma}\label{l-conf}
Let $D$ be a component of $V$ and $t>0$. Then, for every $t>0$, the set $\mathcal{H}_{t,D}$ of restrictions $f|_D$ over all $f\in\mathcal{I}_V$ such that
$||f'||(x)\ge t$ for some $x\in K\cap D$ is finite.
\end{lemma}
\begin{proof}
Fix a subdomain $D'$ of $D$ such that $D\cap K\subset D'\subset\overline{D'}\subset D$.
It follows from the Koebe distortion theorem that, for some $\lambda>1$ and
any conformal $h: D'\to \C$,
\begin{equation}\label{distor0}
\frac{||h'||(x)}{||h'||(y)}\le \lambda \mbox{ for all } x,y\in \overline{D'}.
\end{equation}
This implies that there exist $0<s<S<\infty$ such that, for every $f\in\mathcal{H}_t$,
\begin{equation}\label{bd0}
s\le ||f'||(x)\le S \mbox{ for all } x,y\in \overline{D'}.
\end{equation}
Indeed, the left hand-side inequality holds with $s=t/\lambda$. For the other inequality,
assuming the contrary and using (\ref{distor0}), we find a sequence $f_n\in\mathcal{H}_{t,D}$
so that $||f_n'||\to\infty$ uniformly in $D'$. Let $K_D=D\cap K$. As $K_D\subset D'$ and $K_D$ has no isolated points, then
$\diam(K_D)>0$. Therefore, $\diam K\ge \diam f_n(K_D)\to\infty$, a contradiction which proves (\ref{bd0}).

Now, $\mathcal{H}_{t,D}$ splits into two disjoint subsets $\mathcal{H}^{+}_t$, $\mathcal{H}^{-}_t$ consisting of holomorphic and anti-holomorphic maps, respectively.
Therefore, by (\ref{bd0}), $\mathcal{H}^{+}_t$ is a non-trivial normal family on $D'$, hence, finite (remind that $K$ is a J-like non-laminar compact).
To prove that $\mathcal{H}^{-}_t$ is finite,
let us fix some $f_0\in\mathcal{H}^{-}_t$ and consider the set of functions
$\mathcal{H}':=\{f\circ f_0^{-1}: f\in \mathcal{H}^{-}_t\}$ which are well defined on a connected neighborhood $V':=f_0(D')$ of $K':=(f_0)(K_D)$.
Moreover, $\mathcal{H}'$ consists of holomorphic (univalent) functions and (\ref{bd0}) implies that this is a normal non-trivial family on $V'$.
It is easy to see that each $h=f\circ f_0^{-1}\in\mathcal{H}'$ is a symmetry of $K$ on $V'$ because $K\cap V'=K'$ and $h: V'\to f(D')$ is a homeomorphism such that $h(K')=f(D')\cap K$. Hence, $\mathcal{H}'$ is finite by Theorem \ref{t-symm}
which implies that $\mathcal{H}^{-}_t$ is finite too. That is, $\mathcal{H}_{t,D}$ is the union of two finite sets.
\end{proof}
We prove Theorem \ref{t-homeo} by induction on the number $m$ of components of $V$. Note that $m<\infty$ because $K$ is compact and each component of $V$ intersects $K$.
We use the following fact:
given a compact $K$ which is not a singleton, and a neighborhood $V$ of $K$ such that all components $D_1,...,D_m$ of $V$ intersects $K$
there is $T=T(K,V)>0$ as follows: given $f\in\mathcal{I}_V$, there is $x_f\in D_j$
for some $j\in\{1,...,m\}$, such that $||f'||(z_f)\ge T$.
Indeed, otherwise, using the Koebe distortion on each domain $D_i$, we construct a sequence of homeomorphisms of $K$ which tend to a constant function uniformly on $K$, a contradiction.

{\it Basis of induction:} $m=1$, i.e. $V$ is connected.
Then $\mathcal{I}_V=H_{T,V}$ where $T=T(K,V)>0$ and we apply Lemma \ref{l-conf}.

{\it Step of induction:} assume that the theorem holds for every J-like non-laminar compact $K$ with no isolated points and for every neighborhood $V$
of $K$ having at most $m$ components. Let us prove that $\mathcal{I}_V$ is finite for $V$ having $m+1$ components $D_1,...,D_{m+1}$.
So we fix such $K$ and $V$ and find $T=T(K,V)$.
For every $f\in\mathcal{I}_V$ there exist $j\in\{1,...,m+1\}$ and $x_f\in D_j\cap K$ such that $||f'||(x_f)\ge T$.
If we assume, by a contradiction, that the set $\mathcal{I}_V$
contains an infinite sequence of different conformal maps,
then we find some fixed $j\in\{1,...,m+1\}$ and an infinite subsequence $\{f_n\}\subset \mathcal{H}_{T, D_j}$ of this sequence.
Then,
by Lemma \ref{l-conf} the set of restrictions $f_n|_{D_j}$ contains only finitely many different maps, that is, there is a further subsequence
$\{f_{n_k}\}$ and a conformal map
$h:D_j\to\C$ such that $h\equiv f_{n_k}|_{D_j}$ for all $k$. Let
$\tilde K=K\setminus h(K\cap D_j)$ and $\tilde V=\cup_{1\le i\le m+1, i\neq j}D_i$. Then $\tilde V$ is a neighborhood of a J-like non-laminar compact $\tilde K$ which, on the one hand, consists of $m$ components
$D_i$, $i\neq j$, on the other hand, $\mathcal{I}_{\tilde V}$ contains infinitely many different maps
$\tilde f_k$ where $\tilde f_k\equiv f_{n_k}$ on $D_i$, $i\neq j$, a contradiction with the induction assumption.
\subsection{Theorem \ref{t-finite}+Corollary \ref{c-finite}}
{\it Proof of Theorem \ref{t-finite}}.
First, the conditions of items (2)+(3) mean that the marked IFS $G_*$ satisfies either the SSC or the strong OSC with the open set $U$.
In the latter case, let us fix a component $D$ of $U$ which intersects $A$. By parts 1 and 3 of Lemma \ref{l-attrjlike} applied to $G_*$
along with Theorem \ref{t-symm},
we conclude that every non-trivial family of symmetries of $A$ on $D$ is finite (recall that $A$ is non-laminar).
Denote $\mathcal{G}^+=\{g:\Omega_G\to\C: G\in\mathcal{G}\}$.
Fix a smaller connected neighborhood $V'$ of $A$ such that $\overline{V'}\subset V$ and $\overline{U}\subset V'$.
It follows from the Koebe distortion theorem as in the proof of (\ref{bd0}) of Lemma \ref{l-conf} that
there exists $0<S<\infty$ such that, for every $g\in\mathcal{G}^+$ and all $x\in \overline{V'}$,
\begin{equation}\label{bd}
|g'(x)|\le S.
\end{equation}
Next, by the conditions on the set $U$ and by item (2), we can apply part 2 of Lemma \ref{l-attrjlike} and conclude that every $g\in\mathcal{G}^+$ is a symmetry of $A$ on every component of $U$, in particular, on $D$.
Along with (\ref{bd}) this means
that
the family of restrictions
$\mathcal{D}=\{g|_D: g\in \mathcal{G}^+\}$ is a normal family of symmetries of $A$ on $D$.

As $D\subset V$ and $V$ is connected, by the Uniqueness theorem it is enough to show that the family $\mathcal{D}$ is finite. By Theorem \ref{t-symm}, this is reduced to proving that
$\mathcal{D}$ contains no sequence converging to a constant function in $D$. Assuming the contrary, let $\{G_n\}_{n=1}^\infty\subset\mathcal{G}$
do contain a sequence of maps converging to a constant uniformly on compacts in $D$. Note that by (\ref{bd})) and since $\overline{U}\subset V'\subset\overline{V'}\subset V$ the uniform converges on compacts in $D$ is equivalent to the uniform convergence on $\overline{V'}$.
So we say below that a sequence in $\mathcal{D}$ converges to a holomorphic function meaning uniformly on  $\overline{V'}$.
Passing to a subsequence, one can assume that
$\# G_n=m$ for all $n$ and some $m\ge 2$. Let $G_n=(g_1^{(n)},...,g_m^{(n)})$. We arrive at a contradiction by constructing
an infinite subsequence of $G_n$ which is constant, i.e., consisting of the same $m$ functions for all $n$ along the subsequence.

For each $n$, $\cup_{i=1}^m g_i^{(n)}(A)\supset A$. Therefore, there is a converging sequence of maps $\{g_{i_{n_r}}^{(n_r)}\}\subset\mathcal{D}$ which do not tend to a constant, hence,
by Theorem \ref{t-symm}, this sequence is finite. Replacing $\{G_n\}$ by its subsequence $\{G_{n_r}\}$, re-indexing maps in each $G_{n_r}$ if necessary, and re-naming it by (new) sequence $\{G_n\}$,
one can assume that $g_1^{(n)}\equiv g_1$ for all $n$ and some $g_1:V\to \C$. Now, $g_1(\overline{U})=g_1^{(n)}(\overline{U})$ is disjoint with $\cup_{i=2}^m g_i^{(n)}(U)$. As $A\subset \overline{U}$, this implies that $g_1(A)$ is a proper subset of $A$, i.e.,
$A_1:=A\setminus g_1(A)\neq\emptyset$. But $\cup_{i=2}^m g_i^{(n)}(A)\supset A_1$. This allows us to repeat the procedure with the new $\{G_n\}$ arriving at its subsequence (again re-named by $\{G_n\}$)
such that $g_1^{(n)}\equiv g_1$ and $g_2^{(n)}\equiv g_2$ for all $n$. As $g_1(\overline{U})\cup g_2(\overline{U})=g_1^{(n)}(\overline{U})\cup g_2^{(n)}(\overline{U})$ is disjoint with $\cup_{i=3}^m g_i^{(n)}(U)$, we get
$A_2:=A\setminus (g_1(A)\cup g_2(A))\neq\emptyset$, and can repeat the procedure again. After $m$ steps we construct a constant subsequence of the original sequence $\{G_n\}$, a contradiction.

{\it Proof of Corollary \ref{c-finite}}.

For the family $\mathcal{G}$ that satisfies item 1 and either item 2a or item 2b, we construct a domain $U$ as in Theorem \ref{t-finite} and then check all the conditions of the latter to get the desired conclusion.

Start by noting that in both cases, item 1 guarantees (as in the proof (\ref{bd0}) of Lemma \ref{l-conf}
that, for every domain $V'$ with $\overline{V'}\subset V$
there exists $S>1$ such that $|g'(x)|\le S$ for all $x\in\overline{V'}$ and any $g\in \mathcal{G}^+$.

{\it So assume we have 2a}. Reducing $r$ if necessary, one can assume that $\dist(\partial V, A)>3r$. For $V'$ to be the $2r$-neighborhood of $A$,
find the constant $S>1$ and define $U$ to be the $r/S$-neighborhood of $A$. Then $g(U)$ are pairwise disjoint for all $g\in G$ and every $G\in\mathcal{G}$. Then all the conditions of Theorem \ref{t-finite} are satisfied: (3) is ok with just constructed set $U$, (1) holds because of the bound (\ref{2a}) and (4) holds because every $G\in\mathcal{G}$ has the SSC.

{\it Now assume that the conditions of 2b hold}.

Let $[a,b]$ be the minimal closed interval containing $A$. The SSC implies that, for any $G\in\mathcal{G}_\R$ and any two different
$g,\tilde g\in G$, we have:
\begin{equation}\label{disj}
g(I)\cap \tilde g(I)=\emptyset \mbox{ where } I=(a,b).
\end{equation}
Let's apply Lemma \ref{l-sull} and find two angles $0<\theta<\theta'<\pi/2$ such that
$g(\Omega_\theta(I))\subset \Omega_{\theta'}(g(I))$ for all
$g\in \mathcal{G}^+:$. Note that $\Omega_{\theta'}(J)\cap \Omega_{\theta'}(\tilde J)$ for any two disjoint open
intervals $J, \tilde J$ because $\theta'<\pi/2$. Therefore, we can let
$$U=\Omega_\theta(I)$$
to satisfy the condition (2) of Theorem \ref{t-finite}, while conditions (\ref{sup}), (1) and (3) of Theorem \ref{t-finite} automatically hold because they are assumed in the corollary (in the considered case 2b). Thus Theorem \ref{t-finite} applies in this case too.
\subsection{Theorems \ref{t-symmifs}-\ref{t-finitegen}: the key lemma}
Theorems \ref{t-symmifs}-\ref{t-finitegen} fairly easily follow from Theorem \ref{t-symm} and the next
\begin{lemma}\label{l-symmifs}
Let $G=\{g_i\}_{i\in\Lambda}$ and
$F=\{f_j\}_{j\in\tilde\Lambda}$ be two finite holomorphic IFS with the SSC which are defined on domains $\Omega_G$ and $\Omega_F$ respectively.
Suppose they
share the same attractor $A$.
Let $\rho:=\min (\rho_G, \rho_F)$
and $N$ be so that for any $n\ge N$ and any ${\bf w}\in\Lambda^n$,
$$\sup\{|g_{\bf w}'(x)|: x\in A\}\le s_F$$
where
$$s_F=\inf\{|f'_j(y)|: j\in\tilde\Lambda, y\in A\}.$$
(Such $N$ exists because $g_{\bf w}'(x)\to 0$ as $|{\bf w}|\to\infty$ uniformly in ${\bf w}\in\Lambda^*$ and $x\in A$. Note also that
$s_F<1$.)

Let $r=(3-\sqrt{8})\rho$. There exists
a function ${\bf v}:A \times \cup_{n\ge N}\Lambda^n\to \cup_{k=1}^\infty \tilde\Lambda^k$ as follows. First,
\begin{equation}\label{wvinf}
\lim_{|{\bf w}|\to\infty} |{\bf v}(a, {\bf w})|=\infty
\end{equation}
uniformly in $a\in A$, ${\bf w}\in\Lambda^*$.
For every $a\in A$ and ${\bf w}\in\Lambda^n$ with $n\ge N$, the map
$$H_{{\bf w}, {\bf v}(a,{\bf w})}:=f_{{\bf v}(a,{\bf w})}^{-1}\circ g_{\bf w}$$
is defined and univalent in $B(a, r)$ and such that
\begin{equation}\label{lsymmifs1}
B(H_{{\bf w}, {\bf v}(a,{\bf w})}(a), \frac{s_F}{25}\rho)\subset H_{{\bf w}, {\bf v}(a,{\bf w})}(B(a,r))\subset B(H_{{\bf w}, {\bf v}(a,{\bf w})}(a),\rho).
\end{equation}
Furthermore, $H_{{\bf w}, {\bf v}(a,{\bf w})}$ is a symmetry of $A$ on $B(a, r)$.
Finally, there exists
$K=K(F, G)<\infty$
such that, for every $a\in A$,
the number of different functions $H_{{\bf w}, {\bf {v}}(a,{\bf w})}$ in $B(a, r)$, over all ${\bf w}\in\Lambda^*$ with $|{\bf w}|\ge N$ is at most $K$.
\end{lemma}
\begin{proof} The following construction is a somewhat similar to \cite{le}, \cite{lp}.
For $\rho=\min (\rho_G, \rho_F)$, all maps $g_{\bf w}$, $f_{\bf v}$, $|{\bf w}|, |{\bf v}|<\infty$, are local symmetries on $B(a,\rho)$ for every $a\in A$.
Let $n\ge N$, ${\bf w}\in\Lambda^n$ and $a\in A$.
Let $a({\bf w})=g_{\bf w}(a)$. Since $a({\bf w})$ is a point of the attractor $A$ of $F$ and the SSC holds for $F$,
there exists a unique infinite word $(j_1,j_2,...,j_k,...)\in\tilde\Lambda^*$ such that, for each
finite $V_k=(j_1,...,j_k)$
there is $b_k\in A$ so that $f_{V_k}(b_k)=a(w)$.
Hence, all maps $f_{V_k}:B(b_k,\rho)\to\C$ are well-defined and $f_{V_k}(b_k)=a({\bf w})$.
Note that $b_k=f_{j_{k+1}}(b_{k+1})$ and $f_{V_{k+1}}=f_{V_k}\circ f_{j_{k+1}}$.
As $|g'_{\bf w}(a)|\le s_F\le |f'_{v_1}(b_1)|$ and $f'_{V_k}(b_k)\to 0$ as $k\to\infty$, the following number is well-defined:
$$m=\max\{k\ge 1: |f'_{V_k}(b_k))|\ge |g'_{\bf w}(a)|\}.$$
Observe that
$$|g'_{\bf w}(a)|\le |f'_{V_m}(b_m)|\le \frac{1}{s_F}|g'_{\bf w}(a)|,$$
by the definition of $s_F$ and $m$ (as if the right-hand side inequality breaks down, we get a contradiction with the maximality of $m$).
Let us show that  the map $f_{V_m}^{-1}\circ g_{\bf w}$ is well-defined on $B(a,r)$.
Indeed,
by Theorem \ref{koebe}, on the one hand,
$$f_{V_m}(B(b_m,\rho))\supset B(a(w),\frac{1}{4}\rho |f'_{V_m}(b_m)|)$$
while on the other hand,
$$g_{\bf w}(B(a,r))\subset B(a({\bf w}),\frac{1}{4}\rho |g'_{\bf w}(a)|)$$
where we use that $t/(1-t)^2=1/4$ for $t=3-\sqrt{8}$.
Since $|f'_{V_m}(b_m))|\ge |g_{\bf w}(a)|$, this shows that
$$g_{\bf w}(B(a,r))\subset f_{V_m}(B(b_m,\rho)).$$
Now, let ${\bf v}(w,a)=V_m$ and $H_{{\bf w},{\bf v}(a,{\bf w})}=f_{V_m}^{-1}\circ g_{\bf w}$.
Thus indeed we have a well-defined univalent map
$$H_{{\bf w},{\bf v}(a,{\bf w})}:B(a,r)\to B(b_m,\rho).$$
We have: $H_{{\bf w},{\bf v}(a,{\bf w})}(a)=b_m$ and
\begin{equation}\label{deriv}
|H'_{{\bf w},{\bf v}(a,{\bf w})}(a)|=\frac{|g'_{\bf w}(a)|}{|f_{V_m}(b_m)|}\in [s_F, 1].
\end{equation}
Hence, again by the Koebe distortion theorem \ref{koebe},
$$H_{{\bf w},{\bf v}(a,{\bf w})}(B(a,r))\supset B(b_m, \frac{1}{4}s_F r)$$
where $r/4=(3-\sqrt{8})\rho/4>\rho/25$.
The map $H:=H_{{\bf w}, {\bf v}(a,{\bf w})}=f_{V_m}^{-1}\circ g_{\bf w}$ is a symmetry of $A$ on $B(a,r)$ because
$H(B(a,r))\subset B(b_m, \rho)$ while $f_{V_m}$ is a symmetry of $A$ on $B(b_m, \rho)$ by the choice of $\rho$.

Let's prove the final claim.
By (\ref{lsymmifs1}), for every $a\in A$, the family $\mathcal{H}_a:=\{H_{{\bf w}, {\bf v}(a,{\bf w})}\}$ is a normal nontrivial
family of symmetries of $A$ on $B(a, r)$, hence, by Theorem \ref{t-symm}, the number of different functions in $\mathcal{H}_a$
is bounded by some $K_a<\infty$. Let us prove that $\sup_{a\in A} K_a<\infty$. Indeed, otherwise we find a sequence $a_n\to a_*$ such that
$K_{a_n}\to\infty$. Then, since $B(a_*,r/2)\subset B(a_n,r)$ for all $n$ large, the number of different symmetries $H_{{\bf w}, {\bf v}(a_n,{\bf w})}$,
$|{\bf w}|>N$, in $B(a_*, r/2)$ is bigger than $K_{a_n}$ for any big $n$, i.e., infinite. On the other hand, as $A$ is a Cantor J-like compact,
this number has to be bounded, a contradiction. Note finally that (\ref{wvinf}) follows directly from the construction.
\end{proof}
\subsection{Proof of Theorem \ref{t-symmifs} and Corollary \ref{c-preper}}\label{ss-main}
First, we show that (I) and (II) are equivalent. Then we prove that
$(I)\rightarrow(III)
\rightarrow(I)$.

(I)$\leftrightarrow$(II):
In one direction, let us prove that (II) implies (I). As $F$ and $G$ can be switched, it is enough to prove that $A_G\subset A_F$ assuming (G) holds where
$A_G, A_F$ are attractor of $G$, $F$ respectively.
Let $a\in A_G$. There exists an {\it infinite word} ${\bf w}=(i_1,...,i_n,...)\in\Lambda^\N$ such that $g_{{\bf w}_n}\to a$ as $n\to\infty$ uniformly in $\Omega_G$, where ${\bf w}_n=(i_1,...,i_n)$.
Passing to a subsequence ${\bf w}_{n_r}$ we find a sequence ${\bf v}_r\to\infty$ and a fixed $k\in\{1,...,k_*\}$ such that
$g_{{\bf w}_{n_r}}\circ g_{{\bf w^k}}^{-1}(z)=f_{{\bf v}_r}\circ f_{{\bf v^k}}^{-1}(z)$ for all $z\in g_{{\bf w^k}}(B_G)$. Fix such $z$.
Since the point $x:=f_{{\bf v^k}}^{-1}(z)$ is a fixed point of $\Omega_F$ and $|{\bf v}_r|\to\infty$, then $f_{{\bf v}_r}(x)\to A_F$, hence, $a\in A_F$.

In the opposite direction, assume that $A=A_G=A_F$, fix $\rho>0$ as in Lemma \ref{l-symmifs} and let $B=B_G=B_F=B(a, r)$, for some $a\in A$
and $r$ as in Lemma \ref{l-symmifs}.
Then apply
Lemma \ref{l-symmifs}: for every finite word ${\bf w}\in\Lambda^*$ with $|{\bf w}|\ge N$ there is a finite word ${\bf v}(a, {\bf w})\in \tilde\Lambda^*$ such that
$H_{{\bf w}, {\bf v}(a,{\bf w})}:=f_{{\bf v}(a,{\bf w})}^{-1}\circ g_{\bf w}$ is well-defined in $B$ and the family of maps
$\{H_{{\bf w}, {\bf v}(a,{\bf w})}\}$ is finite.
That is, there are only finitely many different among them: there is a finite collection of words
${\bf w^k}, {\bf v^k}:={\bf v}(a,{\bf w^k})$, $k=1,...,k_0$,
and some $m_0\ge N$ such that, for each finite word ${\bf w}\in\Lambda^*$ with $|{\bf w}|\ge m_0$ there is $k\in\{1,...,k_0\}$ for which
$H_{{\bf w}, {\bf v}(a,{\bf w})}=H_{{\bf w^k}, {\bf v}(a,{\bf w^k})}$, i.e.
$f_{{\bf v}(a,{\bf w})}^{-1}\circ g_{\bf w}=f_{{\bf v^k}}^{-1}\circ g_{{\bf w^k}}$ identically on $B$, equivalently, the equality
$f_{\bf v}\circ f_{{\bf v^k}}^{-1}=g_{\bf w}\circ g_{{\bf w^k}}^{-1}$ holds on $g_{{\bf w^k}}(B)$, where
${\bf v}={\bf v}(a,{\bf w})$.

(I)$\rightarrow$(III). So assume that $A=A_G=A_F$.
Let $\rho$, $r$ be as in Lemma \ref{l-symmifs}.
By Lemma \ref{l-ifsinv}(2), there exist box-like restrictions $(G, V_G)$ of $(G, \Omega_G)$ and $(F, V_F)$ of $(F, \Omega_F)$
such that every component of either $V_G$ or $V_F$ has the Euclidean diameter less than $r$.
In each (of finitely many) connected component $D_i$ of $V_G$ choose a point $a_i\in A$
and let $B_i=B(a_i, r)$. To each $B_i$ we apply Lemma \ref{l-symmifs} and find some finite collection of words ${\bf t}_{i,j}$,
${\bf u}_{i,j}={\bf v}(a_i, {\bf t}_{i,j})$, such that any $H_{{\bf w}, {\bf v}(a_i,{\bf w})}$ with
$|{\bf w}|$ big enough
is equal to some $H_{{\bf t}_{i,j}, {\bf u}_{i,j}}$
on $B_i$. As $H_{{\bf t}_{i,j}, {\bf u}_{i,j}}$ is a symmetry of $A$ on $B(a_i,r)$ and $D_i\subset B(a_i, r)$, then
\begin{equation}\label{iii}
H_{{\bf t}_{i,j}, {\bf u}_{i,j}}(A\cap D_i)\subset A\subset V_F.
\end{equation}
Let $\Gamma_i\subset \N\times\N$ be a finite set of pairs $(i,j)$ for all words ${\bf t}_{i,j}, {\bf u}_{i,j}$ obtained in this way
for each component $D_i$,
and $\Gamma_G=\cup_i\Gamma_i$. Let $I_G=\{{\bf t}_{i,j}, {\bf u}_{i,j}: (i,j)\in\Gamma_G\}$
and $M_G=\max\{|{\bf t}_{i,j}|, |{\bf u}_{i,j}|: (i,j)\in\Gamma_G\}$. By (\ref{wvinf}) of Lemma \ref{l-symmifs}, there exists $m_G\ge M_G+1$ such that,
for any $i$, if
$|{\bf w}|\ge m_G$ then $|{\bf v}(a_i,{\bf w})|\ge M_G+1$. Thus for every component $D_i$ of $V_G$ and every ${\bf t}\in\Lambda^{m_G}$ there exists ${\bf u}\in\tilde\Lambda^*$,
$|{\bf u}|\ge M_G+1$, and a pair $({\bf t}_{i,j}, {\bf u}_{i,j})$ for some $(i,j)\in\Gamma_G$ as follows:
$f_{\bf u}^{-1}\circ g_{\bf t}=f_{{\bf u}_{i,j}}^{-1}\circ g_{{\bf t}_{i,j}}$ identically on $D_i\cap A$, equivalently, the equality
$$f_{\bf u}\circ f_{{\bf u}_{i,j}}^{-1}=g_{\bf t}\circ g_{{\bf t}_{i,j}}^{-1}$$
holds on $g_{{\bf t}_{i,j}}(D_i\cap A)$ where, by (\ref{iii}), the latter set is a subset of
$f_{{\bf u}_{i,j}}(V_F)$.

Now we do the same switching $G$ and $F$
and finding the corresponding finite collection $I_F=\{{\bf \tilde t}_{i,j}, {\bf \tilde u}_{i,j}: (i,j)\in\Gamma_F\}$ of words,
and the corresponding numbers $M_F$, $m_F$.
The union $I_G\cup I_F$ is (after re-indexing it) the required set of words $\{{\bf t^k}, {\bf u^k}\}_{k=1}^K$, for some $K\in\N$,
and $M=\max\{m_G, m_F\}$. (Notice that there can be repetitions of words in $I_G\cup I_F$.)
Then (III) holds.

(III)$\rightarrow$(I).
It is enough to show that the system of equations (\ref{func-finite}) implies that $A_G\subset A_F$
(the opposite inclusion would follow after switching $G$ and $F$).
Let $\varphi=\varphi_G:U_G\to V_G$, $\psi=\varphi_F:U_F\to V_F$ be the associated to $G$, $F$ inverse complex maps, see Subsection \ref{ss-sep}. As pre-periodic points
of $\varphi$, $\psi$ are all in $A_G$, $A_F$ respectively and, moreover, are dense there,
it would be enough to prove that any pre-periodic point of $\varphi$ is a pre-periodic point of $\psi$, i.e.,
if $x\in A_G$ and the sequence $\{\varphi^n(x)\}_{n\ge 0}$ is finite, then $\{\psi^n(x)\}_{n\ge 0}$ is well-defined and finite too.
For this purpose, we use the system of eq. (\ref{eq-inv}) which is equivalent to (\ref{func-finite}). For convenience, let us repeat it here:
for every component $D$ of $V_G$ and every ${\bf t}\in\Lambda^M$
there exists some ${\bf u}={\bf u}(D_G,{\bf t})\in\tilde\Lambda^*$, $|{\bf u}|\ge M_*+1$,
such that $g_{\bf t}(D\cap A_G)\subset f_{\bf u}(V_F)$ and
\begin{equation}\label{eq-inv1}
\psi^{-|{\bf u^k}|}\circ \psi^{|{\bf u}|}(x)=\varphi^{-|{\bf t^k}|}\circ \varphi^M(x) \mbox{ for every } x\in
g_{\bf t}(D\cap A_G).
\end{equation}
Let $D_i$, $i=1,...,\nu$, be all components of $V_G$.
Denote for brevity $A=A_G$ and let $A_i=A\cap D_i$, $1\le 1\le \nu$. Observe that
$A$ is a disjoint union of sets $\{g_{\bf t}(A_i): {\bf t}\in\Lambda^M, 1\le i\le \nu\}$. This implies that
given $x\in A$ there exist unique ${\bf t}={\bf t}(x)\in\Lambda^M$ and $i\in\{1,...,\nu\}$ such that $x\in g_{\bf t}(A_i)$. Then
the following map
$S: A\to A$ is well-defined and has two expressions:
$$S(x):=\varphi^{-|{\bf t^k}|}\circ \varphi^M(x)=\psi^{-|{\bf u^k}|}\circ \psi^{|{\bf u}|}(x), \ \ x\in A,$$
where ${\bf u}={\bf u}(x)$ and $k=k(x)$.
Let $\{x_i\}_{i=0}^n$ where $x_1=S(x_0),...,x_n=S(x_{n-1})$, be $n$ iterates by $S$ of some $x_0\in A$. Using the first formula for $S$, we get:
$$x_n=\varphi^{-|{\bf t^{k(x_{n-1})}}|}\circ \varphi^{nM-\sum_{i=0}^{n-2}|{\bf t^{k(x_i)}}|}(x_0).$$
Using the second expression for $S$, we similarly get:
$$x_n=\psi^{-|{\bf u^{k(x_{n-1})}}|}\circ \psi^{\sum_{i=0}^{n-1}|{\bf u}(x_i)|-\sum_{i=0}^{n-2}|{\bf u^{k(x_i)}}|}(x_0).$$
It is important to have in mind that
$M\ge M_*+1$, $|{\bf u}(x)|\ge M_*+1$ for all $x\in A$, where $M_*=\max\{|{\bf t^k}|, |{\bf u^k}|, 1\le k\le K\}$.
Hence, for all $n\ge 1$,
$$M_n:=nM-\sum_{i=0}^{n-2}|{\bf t^{k(x_i)}}|\ge n$$
and
$$N_n:=\sum_{i=0}^{n-1}|{\bf u}(x_i)|-\sum_{i=0}^{n-2}|{\bf u^{k(x_i)}}|\ge n.$$
Thus for any $x_0\in A$ and any $n>0$, we have two formulas for $x_n$:
$$x_n=\varphi^{-|{\bf t^{k(x_{n-1})}}|}\circ \varphi^{M_n}(x_0)=\psi^{-|{\bf u^{k(x_{n-1})}}|}\circ \psi^{N_n}(x_0),$$
where $M_n, N_n\to\infty$ as $n\to\infty$ and branches
$$\varphi^{-|{\bf t^{k(x_{n-1})}}|}=g_{{\bf t^{k(x_{n-1})}}}, \ \ \psi^{-|{\bf u^{k(x_{n-1})}}|}=f_{{\bf u^{k(x_{n-1})}}}$$
run over only finitely many maps $g_{{\bf t^k}}$, $f_{{\bf u^k}}$, $k=1,...,K$.
Now, let $x_0$ be a pre-periodic point of $\varphi$. Then, by the first formula for $x_n$,
the sequence $\{x_n\}_{n=0}^\infty$ takes only finitely many different values. Hence, by the second formula for $x_n$,
the sequence $\{\psi^{N_n}(x_0)\}_{n=0}^\infty$ is well-defined and also takes only finitely many different values.
As $N_n\to\infty$, there are some $0<P<Q$ such that $\psi^P(x_0)=\psi^Q(x_0)$, i.e., $x_0$ is a pre-periodic point of $\psi$ as well.
This finishes the proof of the implication (III)$\rightarrow$(I) of Theorem \ref{t-symmifs} and the proof of Corollary \ref{c-preper}.
\subsection{Proof of Theorem \ref{t-finitegen} and Corollaries \ref{c-exact} and \ref{c-sp}}
We begin with a remark. Let $A$ be the attractor of IFS $F=\{f_j\}_{j\in\tilde\Lambda}$ that satisfies SSC.
Suppose that ${\bf v},{\bf \tilde v}$ are two finite words such that, for a point $x\in A$, $f_{\bf v}\circ f^{-1}_{{\bf \tilde v}}(x)=x$. Then we claim that either ${\bf v}={\bf \tilde v}$,
or ${\bf v},{\bf \tilde v}$ have different lengths and the longer of them is a continuation of the shorter one.
Indeed, because the SSC holds for $F$, the fact (used already in the proof of Lemma \ref{l-symmifs}) is that given a point $x\in A$ there exists a unique infinite word $v_\infty(x)$ such that whenever $f_{\hat v}(y)=x$, for some $\hat v$, then $\hat v$ is a truncated finite word of $v_\infty(x)$.
Since $f_{\bf v}\circ f^{-1}_{{\bf \tilde v}}(x)=x$ and  $f_{{\bf \tilde v}}\circ f^{-1}_{{\bf v}}(x)=x$, it follows that both ${\bf v}, {\bf \tilde v}$ are finite truncations
of the same infinite word $v_\infty(x)$, as claimed.

Now, for two holomorphic IFS $F$ and $G$ satisfying the SSC and
having the same attractor $A$
choose $\rho>0$ such that maps $g_{\bf w}$, $f_{\bf v}$, for all finite words ${\bf w}$, ${\bf v}$,  are local symmetries on $B(a, \rho)$, $a\in A$, for all $a\in A$.
Fix any finite word ${\bf w}$ for $G$ and consider iterates $g_{\bf w}^k$, $k=1,2,...$, of the map $g_{\bf w}$ in some disk $B(a,\rho)$, $a\in A$. By Lemma \ref{l-symmifs}, for every $k$
large enough, there is a finite word ${\bf v}_k$ for $F$ such that the map
$H_k:=f_{{\bf v}_k}^{-1}\circ g_{\bf w}^k$ is a well-defined symmetry of $A$ on $B=B(a,r)$ with $r=(3-\sqrt{8})\rho$ which form a non-trivial normal sequence of
symmetries. By Theorem \ref{t-symm} and Lemma \ref{l-symmifs} (where, in particular, the finite number $K=K(F, G)$
was defined), there are $k$ and $l\in\{1,..., K\}$
such that
$H_k=H_{k+l}$ on the ball $B$. That is:
$$f_{{\bf v}_k}^{-1}\circ g_{\bf w}^k=f_{{\bf v}_{k+l}}^{-1}\circ g_{\bf w}^l\circ g_w^k$$
on $B$. Hence,
$$g_{\bf w}^l=f_{{\bf v}_{k+l}}\circ f_{{\bf v}_k}^{-1}$$
on $B_1:=g_{\bf w}^k(B)$. Note now that $B_1$ is a domain that contains the fixed point $\beta_{\bf w}$ of $g_{\bf w}$.
In particular,
$f_{{\bf v}_{k+l}}\circ f_{{\bf v}_k}^{-1}(\beta_{\bf w})=\beta_{\bf w}$ where $\beta_{\bf w}$ is a point of the attractor of $F$. Since $F$ satisfies SSC, by the remark at the beginning of the proof, there are $3$ possibilities: (1) ${\bf v}_k={\bf v}_{k+1}$ which is impossible as $g^l$ is not the identity,
(2) ${\bf v}_k$ is the word ${\bf v}_{k+1}$ extended by some other word $V$ which implies that $g^l=f_V^{-1}$. This is also excluded because both maps share the same fixed point $\beta_{\bf w}$ which is attracting for $g_{\bf w}^l$ and repelling for $f^{-1}_V$. Hence, we are left with the third possibility:
$|{\bf v}_{k+1}|>|{\bf v}_k|$ and ${\bf v}_{k+1}$ is ${\bf v}_k$ extended by some word ${\bf \tilde v}$, so that
$$f_{{\bf v}_{k+l}}=f_{{\bf v}_k}\circ f_{{\bf \tilde v}}$$
on $f_{{\bf v}_k}^{-1}(B_1)$.
Therefore, the following conjugacy equation holds on $B_1$:
$$g_{\bf w}^l=f_{{\bf v}_k}\circ f_{{\bf \tilde v}}\circ f_{{\bf v}_k}^{-1}.$$
As $\beta_{\bf w}\in B_1$ and $g_{\bf w}(\beta_{\bf w})=\beta_{\bf w}$, the point $\tilde\beta:=f_{{\bf v}_k}^{-1}(\beta_{\bf w})$
is a fixed point of $f_{{\bf \tilde v}}$. Thus if $\lambda=g_{\bf w}'(\beta_{\bf w})$, then
$\lambda^l=f'_{{\bf \tilde v}}(\tilde\beta)$, i.e., $\lambda^l\in \sigma(F)$. This proves Corollary \ref{c-sp}.

To prove Theorem \ref{t-finitegen}, for any given $g_i\in G$, start above with a single-symbol word ${\bf w}=(i)$
and find corresponding $l, {\bf v}={\bf v}_k$ and ${\bf \tilde v}$. Then
$g_i^l=f_{\bf v}\circ f_{{\bf \tilde v}}\circ f_{\bf v}^{-1}$ and the conclusion of Theorem \ref{t-finitegen} is a direct consequence of the following easy

{\bf Claim}. Let $R$ be a holomorphic map in a neighborhood of $0$ such that $R(0)=0$ and $\lambda=R'(0)\in\{0<|\lambda|<1\}$.
Given $l$ there are precisely $l$ holomorphic near $0$ maps $g$ such that $g(0)=0$ and $g^l=R$.

Indeed, let $K_R$, $K_g$ be Koenigs' linearization functions (see e.g. \cite{cg}) for $R$, $g$ respectively, i.e.
$K_R$, $K_g$ are local holomorphic injections near $0$ that both fix $0$ and
$R(z)=K_R(\lambda K_R^{-1}(z))$, $g(z)=K_g(\nu K_g^{-1}(z))$ where $\nu^l=\lambda$. Since $K_R, K_g$ are uniquely defined after
normalization $K'_R(0)=K'_g(0)=1$ and $g^l(z)=K_g(\nu^l K_g^{-1}(z))=K_g(\lambda K_g^{-1}(z))=R(z)=K_R(\lambda K_R^{-1}(z))$ where
$|\lambda|\neq 1$,
one gets $K_g=K_R$. Hence, $g(z)=K_R(\nu K_R^{-1}(z))$ where $\nu$ takes $l$ values of $\lambda^{1/l}$. The claim, hence, Theorem \ref{t-finitegen} follow.

\end{document}